\documentclass[10pt]{amsart}
\usepackage[pagebackref]{hyperref}

\font\greekbold=eurb10
\def\bfUpsilon{\hbox{\greekbold\char7}}
\def\bfPhi{\hbox{\greekbold\char8}}
\def\bfOmega{\hbox{\greekbold\char10}}
\def\bfsigma{\hbox{\greekbold\char27}}
\def\bftau{\hbox{\greekbold\char28}}
\def\bfomega{\hbox{\greekbold\char33}}

\renewcommand{\d}{\mathrm{d}}

\newcommand{\eds}{\textsc{eds}}
\renewcommand{\Re}{\operatorname{Re}}
\renewcommand{\Im}{\operatorname{Im}}
\newcommand{\ts}{\textstyle }
\newcommand{\bbR}{{\mathbb R}}
\newcommand{\bbC}{{\mathbb C}}

\newcommand{\iC}{{\mathrm i}}
\newcommand{\p}{\partial}
\newcommand{\GL}{\operatorname{GL}}

\newcommand{\SO}{\operatorname{SO}}
\newcommand{\SU}{\operatorname{SU}}
\newcommand{\Or}{\operatorname{O}}
\newcommand{\Spin}{\operatorname{Spin}}

\DeclareMathOperator{\tr}{tr}

\DeclareMathOperator{\Hom}{Hom}

\newcommand{\Ss}{\mathsf{S}}

\newcommand{\eusu}{\operatorname{\mathfrak{su}}}
\newcommand{\euso}{\operatorname{\mathfrak{so}}}
\newcommand{\Gtwo}{\mathrm{G}_2}

\newcommand{\cF}{{\mathcal F}}

\newcommand{\cI}{{\mathcal I}}

\newcommand{\nb}{\mathbf{n}}

\newcommand{\w}{{\mathchoice{\,{\scriptstyle\wedge}\,}{{\scriptstyle\wedge}}
      {{\scriptscriptstyle\wedge}}{{\scriptscriptstyle\wedge}}}}
\newcommand{\lhk}{\mathbin{\hbox{\vrule height1.4pt width4pt depth-1pt 
             \vrule height4pt width0.4pt depth-1pt}}}

\newcommand{\be}{\begin{equation}}
\newcommand{\ee}{\end{equation}}
\newcommand{\bpm}{\begin{pmatrix}}
\newcommand{\epm}{\end{pmatrix}}

\numberwithin{equation}{section}

\newtheorem{theorem}{Theorem}

\theoremstyle{remark}

\newtheorem{remark}{Remark}

\begin{document}

\author[R. Bryant]{Robert L. Bryant}
\address{Department of Mathematics\\
         University of California\\
         Berkeley, CA 94720-3840}
\email{\href{mailto:bryant@math.berkeley.edu}{bryant@math.berkeley.edu}}
\urladdr{\href{http://www.math.berkeley.edu/~bryant}%
         {http://www.math.berkeley.edu/\lower3pt\hbox{\symbol{'176}}bryant}}

\title[Nonembedding in special holonomy]
      {Nonembedding and nonextension results\\
          in special holonomy}

\date{December 25, 2007}

\begin{abstract}
Constructions of metrics with special holonomy 
by methods of exterior differential systems are reviewed
and the interpretations of these construction 
as `flows' on hypersurface geometries are considered.  

It is shown that these hypersurface `flows' 
are not generally well-posed for smooth initial data 
and counterexamples to existence are constructed.
\end{abstract}

\dedicatory{
Dedicated to Nigel Hitchin with great admiration,
on the occasion of his 60th birthday
}



\thanks{
Thanks to Duke University for its support via
a research grant and to the NSF for its support
via DMS-0604195.  
\hfill\break
\hspace*{\parindent} 
Published in \textsl{The many facets of geometry}.
A tribute to Nigel Hitchin. Edited by Oscar Garc{\'\i}a-Prada, 
Jean Pierre Bourguignon, and Simon Salamon. pp.~346--367, 
Oxford Univ.~Press, Oxford, 2010.
}

\maketitle

\setcounter{tocdepth}{2}
\tableofcontents

\section{Introduction}\label{sec: intro}

In the early analyses of metrics with special holonomy 
in dimensions~$7$ and~$8$, particularly in regards to their existence
and generality, heavy use was made of the Cartan-K\"ahler theorem,
essentially because the analyses were reduced to the study of 
overdetermined \textsc{pde} systems whose natures were complicated by
their diffeomorphism invariance.  The Cartan-K\"ahler theory
is well-suited for the study of such systems and the local 
properties of their solutions.  However, the Cartan-K\"ahler theory 
is not particularly well-suited for studies of global problems 
for two reasons: 
First, it is an approach to \textsc{pde} that relies entirely
on the local solvability of initial value problems and, 
second, the Cartan-K\"ahler theory 
is only applicable in the real-analytic category.

Nevertheless, when there are no other adequate methods for analyzing
the local generality of such systems, the Cartan-K\"ahler theory is
a useful tool, and it has the effect of focussing attention on the
initial value problem as an interesting problem in its own right.
The point of this article is to clarify some of the existence issues
involved in applying the initial value problem to the problem of
constructing metrics with special holonomy.  In particular, 
the role of the assumption of real-analyticity will be discussed,
and examples will be constructed to show that one cannot generally
avoid such assumptions in the initial value formulations of these
problems.

The general approach can be outlined as follows:  
As is well-known (cf.~\cite{MR916718}),
the problem of understanding the local Riemannian metrics in dimension~$n$
whose holonomy is contained in a specified connected group~$H\subset\SO(n)$ 
is essentially equivalent to the problem of understanding the $H$-structures
in dimension~$n$ whose intrinsic torsion vanishes, or, equivalently,
that are parallel with respect to the Levi-Civita connection 
of the Riemannian metric associated to the underlying~$\SO(n)$-structure.
In this article, an $n$-manifold~$M$ endowed with an $H$-structure~$B\to M$ 
with vanishing intrinsic torsion will be said to be an \emph{$H$-manifold}.%
\footnote{While, strictly speaking, an $H$-manifold is a pair~$(M,B)$, 
it is common to refer to~$M$ as an $H$-manifold 
when the torsion-free~$H$-structure~$B$ can be inferred from context.}
In particular, an $H$-manifold is a manifold~$M$ 
endowed with an $H$-structure~$B$ that is flat to first order.

It frequently happens (as it does for all of the cases to be considered here) 
that~$H$ acts transitively on~$S^{n-1}$ 
with stabilizer subgroup~$K\subset H$ where
$$
K = \bigl(\{1\}{\times}\SO(n{-}1)\bigr)\cap H.
$$  
In this case, an oriented hypersurface~$N\subset M$ in an $H$-manifold~$M$
inherits, in a natural way, a $K$-structure~$B'\to N$.  
Typically, this~$K$-structure will not, itself, be torsion-free 
(unless~$N$ is a totally geodesic hypersurface in~$M$), 
but will satisfy some weaker condition on its intrinsic torsion,
essentially that its intrinsic torsion can be expressed 
in terms of the second fundamental form of~$N$ as a submanifold of~$M$. 
The problem then becomes to determine
whether these weaker conditions on a given $K$-structure~$B'\to N^{n-1}$ 
are sufficient to imply that~$(N,B')$ can be induced
by immersion into an $H$-manifold~$(M,B)$.

In the three cases to be considered in this article, 
in which~$H$ is one of~$\SU(2)\subset\SO(4)$, $\Gtwo\subset\SO(7)$, 
or~$\Spin(7)\subset\SO(8)$, it will be shown that the weaker
intrinsic torsion conditions on a hypersurface structure \emph{are}
sufficient to induce an embedding (in fact, essentially a unique one)
\emph{provided that the given $K$-structure
is real-analytic with respect to some real-analytic structure on~$N$.}
It will also be shown, in each case, 
that there are $K$-structures~$B'\to N^{n-1}$
that satisfy these weaker intrinsic torsion conditions 
that \emph{cannot} be induced by an immersion into an $H$-manifold.
Of course, such structures are not real-analytic 
with respect to any real-analytic structure on~$N$.

The existence of the desired embedding in the analytic case 
is a consequence of the Cartan-K\"ahler theorem and, indeed, 
is implicit in the original analyses of $\Gtwo$- and~$\Spin(7)$-structures
to be found in my 1987 paper~\cite{MR916718}.  The examples constructed below, 
showing that existence can fail when one does not have real-analyticity, 
appear to be new.

There is another interpretation of the initial value problem 
that has been considered by a number of authors, 
in particular, Hitchin~\cite{MR1871001}:
The idea of a `flow' of $K$-structures 
that gives rise to a torsion-free $H$-structure.
Let~$M$ be a connected $n$-manifold endowed with an $H$-structure 
and let~$g$ be the underlying metric. 
Let~$N\subset M$ be an embedded, normally oriented hypersurface, 
and let $r:M\to[0,\infty)$ be the distance (in the metric~$g$) 
from the hypersurface~$N$.
As is well-known, there is an open neighborhood~$U\subset M$ of~$N$
on which there is a smooth function~$t:U\to\bbR$ satisfying~$|t|=r$
and~$\d t\not=0$ on~$U$ as well as the condition that its gradient
along~$N$ is the specified oriented normal.  There is then a well-defined
smooth embedding~$(t,f):U\to\bbR\times N$, where, for~$p\in U$, 
the function~$f(p)$ is the closest point of~$N$ to~$p$.  In this way, 
$U$ can be identified with an open neighborhood of~$\{0\}\times N$ 
in~$\bbR\times N$ and, in particular, each level set of~$t$ in~$U$ 
can be identified with an open subset of~$N$.  (When~$N$ 
is compact, there will be an~$\epsilon>0$ such that the level sets~$t=c$
for~$|c|<\epsilon$ will be diffeomorphic to~$N$.)  Thus, at least locally,
one can think of the $H$-structure on~$U$ 
as a $1$-parameter family of~$K$-structures on~$N$.
When one imposes the condition 
that the~$H$-structure on~$U\subset\bbR\times N$ be torsion-free, 
this can be expressed as a first-order initial value problem
with the given $K$-structure on~$\{0\}\times N$ as the initial value.
This first-order initial value problem is sometimes described as a `flow',
but this can be misleading, especially if it causes one to think in terms
of parabolic or hyperbolic \textsc{pde}.  

Indeed, the character of this \textsc{pde} problem 
is more like that of trying to use the Cauchy-Riemann
equations~$u_x = v_y $ and~$v_x=-u_y $ to extend a complex-valued
function defined on the imaginary axis~$x=0$ 
to a holomorphic function on a neighborhood of the imaginary axis. 
One knows that a necessary and sufficient condition for being able to
do this is that the given function on the imaginary axis must be real-analytic.
In the cases to be considered in this article, the requirements are not this
strong since there will be cases in which the initial $K$-structure
is not real-analytic and yet a solution of the initial value problem will
exist.  However, as will be seen, real-analyticity is a sufficient
condition.  

For background on the use of exterior differential systems in this article,
the reader might consult~\cite{MR1083148}.

I have included the case of~$\SU(2)$-structures on $4$-manifolds
because of its historical interest (it was the first case of 
special holonomy to be analyzed) and because the algebra is simpler. 
Also, because other approaches, based on the existence of local 
holomorphic coordinates, have been employed in this case, 
there is an instructive comparison to be made between those methods
and the Cartan-K\"ahler approach.  For this reason, I go into
the $\SU(2)$-case in some detail.  I hope that the reader will find
this as interesting as I have.

\section{Beginnings}\label{sec: beginnings}

\subsection{Holonomy}
Let~$(M^n,g)$ be a connected Riemannian $n$-manifold.

\subsubsection{Parallel transport}
To~$g$, one associates its Levi-Civita connection~$\nabla$, 
which defines, for a piecewise-$C^1$ curve $\gamma:[0,1]\to M$, 
a parallel transport
\be
P^\nabla_\gamma:T_{\gamma(0)}M\to T_{\gamma(1)}M,
\ee
which is a linear $g$-isometry between the two tangent spaces. 

\subsubsection{Group structure}
In 1918, J. Schouten~\cite{schouten1918} considered the set
\be
H_x = \left\{P^\nabla_\gamma\ \vrule\ \gamma(0)=\gamma(1)=x\right\} 
\subseteq \Or(T_xM)
\ee
and called its dimension the number of \emph{degrees of freedom} of~$g$.   

It is easy to establish the identities
\be
P^\nabla_{\bar\gamma} = \left(P^\nabla_\gamma\right)^{-1}
\qquad\text{and}\qquad
P^\nabla_{\gamma_2{*}\gamma_1} 
= P^\nabla_{\gamma_2}\circ P^\nabla_{\gamma_1}
\ee
where~$\bar\gamma$ is the reverse of~$\gamma$
and~$\gamma_2{*}\gamma_1$ is the concatenation of paths~$\gamma_1$
and~$\gamma_2$ satisfying~$\gamma_1(1)=\gamma_2(0)$.

Consequently, $H_x\subset \Or(T_xM)$ is a subgroup and
\be
H_{\gamma(1)} 
= P^\nabla_\gamma\ H_{\gamma(0)}\ \left(P^\nabla_\gamma\right)^{-1}.
\ee

In particular, fixing a linear isometry~$u:T_xM\to\bbR^n$,
the conjugacy class of~$H_u = u H_x u^{-1}$ in~$\Or(n)$ is well-defined,
independent of the choice of~$x\in M$ or the isometry~$u:T_xM\to\bbR^n$.
By abuse of terminology, we say that~$H$ is the \emph{holonomy} of the
metric~$g$ if~$H\subset\Or(n)$ is a group conjugate to some 
(and hence any) of the groups~$H_u$.

For later reference, if~$u:T_xM\to\bbR^n$ is fixed, we let
\be
B_u =  \left\{u\circ P^\nabla_\gamma\ \vrule\ \gamma(1)=x\ \right\}.
\ee  
This~$B_u$ is an~$H_u$-subbundle of the orthonormal coframe bundle of~$g$,
i.e., it is an $H_u$-structure on~$M$.  By its very construction, it
is invariant under $\nabla$-parallel translation and, since~$\nabla$ is
torsion-free, it follows that this $H_u$-structure admits a torsion-free
compatible connection.  

Conversely, let~$H\subset\Or(n)$ be a subgroup and let~$B\to M$ 
be an $H$-structure on~$M$.  If~$B\to M$ admits a compatible, torsion-free
connection~$\nabla$, then~$B$ is said to be \emph{torsion-free}.  
In this case,~$\nabla$ (necessarily unique) 
is the Levi-Civita connection of the underlying metric~$g$ on~$M$ 
and~$B$ is invariant under $\nabla$-parallel translation, 
implying that~$H_u\subset H$ for all~$u\in B$.  
Thus, finding torsion-free $H$-structures on~$M$ 
provides a way to find metrics on~$M$ whose holonomy lies in~$H$.

In this article, I will use the term \emph{$H$-manifold} 
to denote an $n$-manifold~$M^n$ 
endowed with a torsion-free $H$-structure~$B\to M$.
(The intended embedding~$H\subset\Or(n)$ is to be understood from context.)

\subsubsection{Cartan's early results}
In 1925, \'E. Cartan~\cite{Cartan1925} made the following statements:
\begin{enumerate}
\item $H_x$ is a Lie subgroup of~$\Or(T_xM)$, 
           connected if~$M$ is simply-connected.
\item If~$H_x$ acts reducibly on~$T_xM$, 
        then~$g$ is locally a product metric.
\end{enumerate}

Cartan's first statement was eventually proved 
(to modern standards of rigor) by Borel and Lichnerowicz~\cite{MR48133},
and the second statement was globalized and proved by G. de~Rham.

\subsection{A nontrivial case}
In dimensions~$2$ and~$3$, the above facts suffice to determine 
the possible holonomy groups of simply-connected manifolds.

Cartan~\cite{Cartan1925}%
\footnote{Especially note Chapitre~VII, Section~II.}
studied the first nontrivial case, namely, $n=4$ 
and~$H_x\simeq \SU(2)$, and observed that such metrics~$g$
\begin{enumerate}
\item have vanishing Ricci tensor,
\item are what we now call `self-dual', and
\item locally (modulo diffeomorphism) depend on $2$ functions of $3$ variables.
\end{enumerate}

While the first two observations are matters of calculation
and/or definition, the third observation is nontrivial.
However, Cartan gave no indication of his proof and, 
to my knowledge, never returned to this example again.

While I cannot be sure, I believe that it is likely that
Cartan had a proof in mind along the following lines.%
\footnote{He had developed all of the tools necessary 
for this proof in his famous series of papers on pseudo-groups 
and the equivalence problem and, using those results, 
it would have been a simple observation for him.}

The associated $\SU(2)$-structure~$B\to M$ of such a metric~$g$ 
satisfies structure equations of the form
\be\label{eq: SU2streqs}
\begin{aligned}
\d\omega&= {}-\theta\w\omega,\\
\d\theta&= {}- \theta\w\theta + R\bigl(\omega{\w}\omega\bigr),\\
\d R &= {}-\theta{.}R + R'(\omega).
\end{aligned}
\ee
where
\begin{enumerate}
\item the tautological $1$-form~$\omega$ takes values in~$\bbR^4$, 
\item the connection~$1$-form~$\theta$ 
   takes values in~$\eusu(2)\subset\euso(4)$,
\item the curvature function~$R$ takes values in~$W_4$, 
the $5$-dimensional (real) irreducible representation of $\SU(2)$
that lies in~$\Hom\bigl(\Lambda^2(\bbR^4),\eusu(2)\bigr)$, and
\item the derived curvature function~$R'$ takes values in~$V_5$, 
the $6$-dimensional complex irreducible representation of $\SU(2)$
that lies in~$\Hom(\bbR^4,W_4)$.
\end{enumerate}
Calculation shows that the subspace~$V_5$ 
is an involutive tableau in~$\Hom(\bbR^4,W_4)$, 
with character sequence~$(s_1,s_2,s_3,s_4) = (5,5,2,0)$.  
Since the last nonzero character of this tableau is~$s_3 = 2$, 
Cartan's generalization of the third fundamental theorem of Lie 
applies to the structure equation~\eqref{eq: SU2streqs} 
to yield his third observation above.

\subsection{The hyperK\"ahler viewpoint}
Riemannian manifolds~$(M^4,g)$ with
$$
H_x\simeq\SU(2)\subset\SO(4)
$$
are nowadays said to be \emph{hyperK\"ahler}, and we understand them
as special cases of K\"ahler manifolds.  In fact, using our 
understanding of complex and K\"ahler geometry, we now arrive at
Cartan's result by a somewhat different route:

Because the subgroup~$\SU(2)\subset\SO(4)$ acts trivially
on the space of self-dual $2$-forms on~$\bbR^4$,
when~$(M^4,g)$ has~$H_x\simeq\SU(2)$, 
then there exist three $g$-parallel self-dual $2$-forms on~$M$, 
say $\Upsilon_1$, $\Upsilon_2$, and~$\Upsilon_3$, 
such that
\be\label{eq: Upsilondefs}
\Upsilon_i\w\Upsilon_j=2\delta_{ij}\,\d V_g\,.
\ee

Conversely, a triple~$\Upsilon = (\Upsilon_1,\Upsilon_2,\Upsilon_3)$
of closed $2$-forms on~$M$ that satisfies
\be\label{eq: Upsilonwedges}
{\Upsilon_1}^2 ={\Upsilon_2}^2 ={\Upsilon_3}^2 \not=0
\qquad\text{while}\qquad
\Upsilon_2\w\Upsilon_3=\Upsilon_3\w\Upsilon_1=\Upsilon_1\w\Upsilon_2=0
\ee
is easily shown to be $g$-parallel and self-dual 
with respect a unique metric on~$M$ for which~\eqref{eq: Upsilondefs} 
holds.  Moreover, the holonomy of~$g$ 
will then preserve the~$\SU(2)$-structure~$B_u$.

Now, given a triple~$(\Upsilon_1,\Upsilon_2,\Upsilon_3)$
of closed $2$-forms on~$M$ satisfying~\eqref{eq: Upsilonwedges},
one can prove (using the Newlander-Nirenberg theorem) 
that each point of~$M$ 
lies in a local coordinate chart~$z = (z^1,z^2):U\to\bbC^2$ 
for which there exists a real-analytic function~$\phi:z(U)\to\bbR$ 
so that
\be\label{eq: Upsilonincoords}
\Upsilon_2 + \iC\,\Upsilon_3 = \d z^1 \w \d z^2
\qquad\text{and}\qquad
\Upsilon_1 = \tfrac12\iC\,\p\bar\p \phi,
\ee
where~$\phi$ satisfies
the elliptic Monge-Amp\`ere equation
\be\label{eq: phiMA}
\det\left(\frac{\p^2\phi}{\p z^i \p \bar{z}^j}\right) = 1
\ee
and the strict pseudo-convexity condition with respect to~$z$ given by
\be\label{eq: phipsdoconvex}
\left(\frac{\p^2\phi}{\p z^i \p \bar{z}^j}\right) > 0.
\ee

In fact, if one fixes an open set~$U\subset\bbC^2$ 
and chooses a smooth, strictly pseudo\-convex function~$\phi:U\to\bbR$ 
satisfying~\eqref{eq: phiMA}, then the formulae~\eqref{eq: Upsilonincoords}
define a triple~$(\Upsilon_1,\Upsilon_2,\Upsilon_3)$ on~$U$
consisting of $g_\phi$-parallel, self-dual $2$-forms where 
\be
g_\phi = \frac{\p^2\phi}{\p z^i \p \bar{z}^j}\,\d z^i\circ \d {\bar z}^j.
\ee
Thus, the holonomy of $g_\phi$ is (conjugate to) a subgroup of~$\SU(2)$
and it is not difficult to show that, 
for a generic strictly pseudoconvex~$\phi$ satisfying~\eqref{eq: phiMA},
the holonomy of~$g_\phi$ is (conjugate to) the full~$\SU(2)$. 

Note that, because~\eqref{eq: phiMA} is 
an analytic elliptic PDE at a strictly pseudoconvex solution~$\phi$, 
all of its strictly pseudoconvex solutions are real-analytic. 
In particular, a Riemannian metric~$g$ on~$M^4$ with holonomy~$\SU(2)$
is real-analytic in harmonic coordinates (as would have followed
anyway from its Ricci-flatness and a result 
of Deturck and Kazdan~\cite{MR644518}). 

Now, one does not normally think of solving an elliptic PDE
by an initial value problem, but, of course, in the analytic
category, there is nothing wrong with such a procedure.  Indeed,
the Cauchy-Kovalewskaya theorem implies that, in this particular
case, one can specify~$\phi$ and its normal derivative along
a hypersurface, say,~$\Im(z^2)=0$, as essentially arbitrary real-analytic
functions (subject only to an open condition that guarantees the
strict pseudoconvexity of the resulting solution) and thereby
determine a unique strictly pseudoconvex solution of~\eqref{eq: phiMA}.
Thus, in this sense, one sees, again, that the `general' metric
with holonomy in~$\SU(2)$ modulo diffeomorphism depends on two
functions of three variables, in agreement with Cartan's claim.

\section{HyperK\"ahler $4$-manifolds}\label{sec: hk4ms}

\subsection{An exterior differential systems proof}
One can use the characterization of~$\SU(2)$ as the stabilizer
of three $2$-forms in dimension~$4$ as the basis of another
analysis of the existence problem 
via an \eds\ (= `exterior differential system').

Let~$M^4$ be an analytic manifold
and let~$\bfUpsilon$ be the tautological $2$-form on~$\Lambda^2(T^*M)$.
Let
\be
X^{17}\subset \bigl(\Lambda^2(T^*M)\bigr)^3
\ee 
be the submanifold 
consisting of triples~$(\beta_1,\beta_2,\beta_3)\in \Lambda^2(T^*_xM)$
such that
\be
{\beta_1}^2 ={\beta_2}^2 ={\beta_3}^2 \not=0,
\quad\text{and}\quad
\beta_1\w\beta_2 =\beta_3\w\beta_1 =\beta_2\w\beta_3 = 0.
\ee
Let~$\pi_i:X\to \Lambda^2(T^*M)$ for~$1\le i\le 3$ denote the
projections onto the three factors.
The pullbacks~$\bfUpsilon_i=\pi_i^*(\bfUpsilon)$ 
define an \eds\ on~$X$
\be
\cI = \{\d\bfUpsilon_1,\d\bfUpsilon_2,\d\bfUpsilon_3\}.
\ee
The basepoint projection~$\pi: X\to M$ makes~$X$ into a bundle over~$M$
whose fibers are diffeomorphic to the $13$-dimensional homogeneous
space~$\GL(4,\bbR)/\SU(2)$ and whose sections~$\sigma:M\to X$ correspond
to the $\SU(2)$-structures on~$M$.

An $\cI$-integral manifold~$Y^4\subset X$ transverse to~$\pi:X\to M$ 
then represents a choice of three closed $2$-forms~$\Upsilon_i$
on an open subset~$U\subset M$ 
that satisfy the algebra conditions needed 
to define an $\SU(2)$-structure on~$U$.

Calculation shows that~$\cI$ is involutive.  In particular,
a $3$-dimensional real-analytic $\cI$-integral manifold~$P\subset X$ 
that is transverse to the fibers of~$\pi$ can be `thickened' 
to a $4$-dimensional integral manifold that is transverse 
to the fibers of~$\pi$.  This `thickening'
will not be unique, however, because of the invariance of the ideal~$\cI$
under the obvious action induced by the diffeomorphisms of~$M$.

\subsection{A sharper result}
Suppose that~$(M^4,g)$ has holonomy~$\SU(2)$ 
and let~$\Upsilon_i$ be three $g$-parallel $2$-forms on~$M$ satisfying
\be
\Upsilon_i\w\Upsilon_j = 2\delta_{ij}\,\d V_g\,.
\ee

If~$N^3\subset M$ is an oriented hypersurface, with oriented normal~$\nb$,
then there is a coframing~$\eta$ of~$N$ defined by
\be
\eta = \begin{pmatrix}\eta_1\\\eta_2\\\eta_3\end{pmatrix} 
= \begin{pmatrix}
  \nb\lhk \Upsilon_1\\\nb\lhk \Upsilon_2\\\nb\lhk \Upsilon_3
  \end{pmatrix}
\ee
and it satisfies
\be
N^*\!\begin{pmatrix}\Upsilon_1\\\Upsilon_2\\\Upsilon_3\end{pmatrix}
= \begin{pmatrix}\eta_2\w\eta_3\\\eta_3\w\eta_1\\\eta_1\w\eta_2\end{pmatrix}
= {\ast}_\eta\eta.
\ee
where~$\ast_\eta$ is the Hodge star associated to the metric~$g_\eta
= {\eta_1}^2+{\eta_2}^2+{\eta_3}^2$ 
and orientation~$\eta_1\w\eta_2\w\eta_3>0$.

In particular, note that the coframing~$\eta$ is not arbitrary, 
but satisfies the system of three first-order PDE
\be
\d({\ast}_\eta\eta)=N^*\d\Upsilon = 0.
\ee

An alternative expression of the involutivity of the system~$\cI$ 
that is better adapted to the initial value problem 
then becomes the following existence and uniqueness result:

\begin{theorem}\label{thm: SU2thicken}
Let~$\eta$ be a real-analytic coframing of~$N$ such that
$\d({\ast}_\eta\eta)=0$.
There exists an essentially unique embedding of~$N$ 
into a $\SU(2)$-holonomy manifold~$(M^4,g)$ 
that induces the given coframing~$\eta$ in the above manner.
\end{theorem}

\begin{remark}[Essential Uniqueness]
The meaning of this term is as follows:  If~$N$ can be embedded into
two different $\SU(2)$-holonomy manifolds~$(M_1,g_1)$ and~$(M_2,g_2)$
in such a way that both embeddings induce the same coframing~$\eta$ on~$N$
by the above pullback formula, then there are open neighborhoods~$U_i\subset M_i$ of the images of~$N$ and a diffeomorphism~$f:U_1\to U_2$ that is
the identity on the image of~$N$ that pulls the metric~$g_2$ back to the
metric~$g_1$.
\end{remark}

\begin{proof} Write~$\d\eta = -\theta\w\eta$ 
where~$\theta=-{}^t\theta$.  (This~$\theta$ exists and is unique
by the Fundamental Lemma of Riemannian geometry.)

On~$N\times\GL(3,\bbR)$ define%
\footnote{In the following formulae, 
I regard~$g:N\times\GL(3,\bbR)\to\GL(3,\bbR)$ as the projection
onto the second factor.  Also, to save writing, I will write~$\eta$
for~$\pi_1^*\eta$ where~$\pi:N\times\GL(3,\bbR)\to N$ is the 
projection onto the first factor.}
\be
\bfomega = g^{-1}\,\eta
\qquad\text{and}\qquad \gamma = g^{-1}\d g + g^{-1}\theta g,
\ee
so that~$\d\bfomega = -\gamma\w\bfomega$.  
On~$X = \bbR\times N\times\GL(3,\bbR)$ define the three $2$-forms
\be
\begin{pmatrix} \bfUpsilon_1\\ \bfUpsilon_2\\ \bfUpsilon_3\end{pmatrix}
= \begin{pmatrix} \d t\w \bfomega_1 + \bfomega_2\w\bfomega_3\\ 
       \d t\w \bfomega_2 + \bfomega_3\w\bfomega_1\\ 
       \d t\w \bfomega_3 + \bfomega_1\w\bfomega_2\end{pmatrix}
= \d t \w \bfomega + {\ast}_{\omega}\bfomega.
\ee
Let~$\cI$ be the ideal on~$X$
generated by~$\{\d\bfUpsilon_1,\d\bfUpsilon_2,\d\bfUpsilon_3\}$.
One calculates
\be
\d\bfUpsilon 
=  \bigl({}^t\gamma-(\tr\gamma) I_3\bigr)\w\,{\ast}_\omega\bfomega
                + \gamma\w\bfomega\w\d t.
\ee
Consequently,~$\cI$ is involutive, with characters~$(s_1,s_2,s_3,s_4)=(0,3,6,0)$.

Since~$\d({\ast}_\eta\eta)=0$, 
the locus~$L = \{0\}\times N\times\{I_3\}\subset X$ 
is a regular, real-analytic integral manifold of the real-analytic 
ideal~$\cI$. (Note that~$L$ is just a copy of~$N$.)

By the Cartan-K\"ahler Theorem, 
$L$ lies in a unique $4$-dimensional $\cI$-integral manifold~$M\subset X$.
The~$\bfUpsilon_i$ thus pull back to~$M$ to be closed and to define the
desired~$\SU(2)$-structure forms $\Upsilon_i$ on~$M$ 
inducing~$\eta$ on~$L=N$. 
\end{proof}

It is natural to ask whether 
it is necessary to assume that~$\eta$ be real-analytic
for the conclusion of~Theorem~\ref{thm: SU2thicken}.
The following result shows that one cannot weaken
this assumption to `smooth' and still get the same conclusion:

\begin{theorem}\label{thm: SU2nonthickening}
If~$\eta$ is a coframing on~$N^3$ that is not real-analytic
in any local coordinate system
and
\be\label{eq: etacoclosedandcmc}
\d({\ast}_\eta\eta)=0
\qquad\text{and}\qquad
{\ast}_\eta({}^t\eta{\w}\d\eta) = 2C
\ee
for some constant~$C$, then~$\eta$ cannot be induced 
by immersing~$N$ into an $\SU(2)$-manifold.

Smooth-but-not-real-analytic coframings~$\eta$ 
satisfying~\eqref{eq: etacoclosedandcmc} do exist locally.
\end{theorem}

\begin{proof} 
Suppose that~$\Upsilon_i$ ($1\le i\le 3$) 
are the parallel $2$-forms on an $(M^4,g)$ with holonomy $\SU(2)$
and let~$N^3\subset M$ be an oriented hypersurface.

Calculation yields that the induced co-closed coframing~$\eta$ 
on~$N$ satisfies
\be
{\ast}_\eta({}^t\eta{\w}\d\eta) = 2 H
\ee
where~$H$ is the mean curvature of~$N$ in~$M$.

Now, since~$g$ is Ricci-flat, 
it is real-analytic in~$g$-harmonic coordinates.
In particular, such coordinate systems can be used to define
a real-analytic structure on~$M$, which is the one that we will 
mean henceforth.  In particular, since the forms~$\Upsilon_i$ 
are~$g$-parallel, they, too,
are real-analytic with respect to this structure. 
If~$H$ is constant, then elliptic regularity 
implies that~$N$ must be a real-analytic hypersurface in~$M$ 
and hence~$\eta$ must also be real-analytic.

Thus, if~$\eta$ is a non-real-analytic coframing on~$N^3$ 
that satisfies~\eqref{eq: etacoclosedandcmc}
for some constant~$C$, then~$\eta$ cannot be induced on~$N$ 
by an embedding into an $\SU(2)$-manifold.

To finish the proof, I will now show how to construct 
a coframing~$\eta$ on an open subset of~$\bbR^3$ 
that is not real-analytic in any coordinate system 
and yet satisfies~\eqref{eq: etacoclosedandcmc}.

To begin, note that, if a coframing~$\eta$ on~$N^3$ 
is real-analytic in any coordinate system at all, 
it will be real-analytic in \emph{$\eta$-harmonic} coordinates, 
i.e., local coordinates~$x:U\to\bbR^3$ satisfying
\be
\d\bigl({\ast}_\eta\d x\bigr) = 0.
\ee
Now, fix a constant~$C$ 
and consider a coframing~$\eta = g(x)^{-1}\,\d x$ on~$U\subset\bbR^3$ 
where~$g:U\to\GL(3,\bbR)$ is a mapping 
satisfying the first-order, quasi-linear system
\be\label{eq: eta7for9}
\d({\ast}_\eta\eta)=0,
\qquad
{\ast}_\eta({}^t\eta{\w}\d\eta) = 2C,
\qquad 
\d\bigl({\ast}_\eta\d x\bigr) = 0.
\ee
The system~\eqref{eq: eta7for9} consists of~$7$ equations 
for the~$9$ unknown entries of~$g$. 

Calculation shows this first-order system to be underdetermined elliptic 
(i.e., its symbol is surjective at every real covector~$\xi$).  
By standard theory, it has smooth local solutions that are not real-analytic.

Taking a non-real-analytic solution~$g$, the resulting~$\eta$ will
not be real-analytic in the $x$-coordinates, which, by construction,
are~$\eta$-harmonic.  Thus, such an~$\eta$ is not real-analytic
in any local coordinate system.
\end{proof}

\begin{remark}[The `flow' interpretation]
The condition~$\d(\d t\w\omega + {\ast}\omega)=0$ 
has sometimes been described as an `$\SU(2)$-flow' 
on coframings of~$N$.  In fact, this closure 
condition can be written in the `evolutionary' form
\be\label{eq: SU2floweqn}
\frac{\d\hfil}{\d t}\,\omega\, = {\ast}_\omega(\d\omega)
-\tfrac12\,{\ast}_\omega({}^t\omega{\w}\d\omega)\,\,\omega.
\ee
By Theorem \ref{thm: SU2thicken}, 
if~$\eta$ on~$N^3$ is real-analytic and satisfies~$\d({\ast}_\eta\eta)=0$, 
then~\eqref{eq: SU2floweqn} has a solution in a neighborhood of~$t=0$
in~$\bbR\times N$ that satisfies the initial condition
\be
{\omega\,}\vrule_{\,t{=}0} = \eta.
\ee

One does not normally think of evolution equations as having to
have real-analytic initial data. 
However, Theorem \ref{thm: SU2nonthickening}
shows that some such regularity assumption must be made.
\end{remark}

\begin{remark}[Coclosed coframings with specified metric]
Given a Riemannian $3$-manifold~$(N,g)$, it is an interesting
question as to when there exist (either locally or globally) 
a $g$-orthonormal coframing~$\eta = (\eta_1,\eta_2,\eta_3)$
that is coclosed.%
\footnote{Of course, \emph{closed} $g$-orthonormal coframings exist 
locally if and only if~$g$ is flat.}

One can formulate this as an \textsc{eds}
for the section~$\eta:N\to B$ of the $g$-orthonormal coframe 
bundle~$B\to N$.  The natural \textsc{eds} for this is not involutive, 
but a slight extension is and is worth describing here.

Let~$\omega_i$ and $\omega_{ij}=-\omega_{ji}$ (where~$1\le i,j\le 3$) 
be the tautological and Levi-Civita connection forms on~$B$.  
In particular, $g$ pulls back to~$B$ to be~${\omega_1}^2{+}{\omega_2}^2{+}{\omega_3}^2$ and these forms satisfy the structure equations
\be\label{eq: Bstreqs}
\d\omega_i=-\omega_{ij}\w\omega_j\,\qquad\
\d\omega_{ij}=-\omega_{ik}\w\omega_{kj}+\tfrac12R_{ijkl}\,\omega_k\w\omega_l\,.
\ee

Because of their tautological reproducing property, one has~$\eta^*(\omega_i)
=\eta_i$.  Consequently, the $g$-orthonormal coframing~$\eta$ is coclosed
if and only if, when regarded as a section~$\eta:N\to B$, 
it is an integral manifold of the \eds\ $\cI_0$ 
generated by the three closed $3$-forms~$\Upsilon_{ij}=\d(\omega_i\w\omega_j)$
with~$1\le i < j \le 3$.  Computation shows that this ideal 
has~$(s_1,s_2,s_3) = (0,2,1)$, while an integral manifold 
on which~$\Omega=\omega_1\w\omega_2\w\omega_3$ is non-vanishing 
(an obvious requirement for a section~$\eta(M)\subset B$) must satisfy
\be\label{eq: I_0integralelements}
\eta^*\omega_{ij} = \epsilon_{ijk} S_{kl}\,\eta^*\omega_l\,,
\ee
where~$\epsilon_{ijk}$ is fully antisymmetric in its indices, 
with~$\epsilon_{123}=1$, and~$S_{kl}=S_{lk}$.  In particular, 
the space of admissible integral elements of~$(\cI_0,\Omega)$ 
at each point of~$B$ has dimension~$6<s_1{+}2s_2{+}3s_3$.  
Hence, the system~$(\cI_0,\Omega)$ is not involutive.

However, the equations~\eqref{eq: I_0integralelements} 
show that any integral of~$(\cI_0,\Omega)$ 
is also an integral of the $2$-form
\be\label{eq: Upsilondefined}
\Upsilon = \omega_1\w\omega_{23}+ \omega_2\w\omega_{31} + \omega_3\w\omega_{12}
         = {\ts\frac12}\epsilon_{ijk}\,\omega_i\w\omega_{jk}\,.
\ee
Moreover, since
\be
\d(\omega_i\w\omega_j) = \epsilon_{ijk}\,\omega_k\w\Upsilon,
\ee
the differential ideal~$\cI$ generated by~$\Upsilon$ contains~$\cI_0$.  
Calculation using~\eqref{eq: Bstreqs} yields
\be
2\,\d\Upsilon = \epsilon_{ijk}\,\omega_i\w\omega_{lj}\w\omega_{lk} - R\,\Omega
\ee
where~$R$ is the scalar curvature of the metric~$g$.  Thus, the ideal~$\cI$
is generated algebraically by~$\Upsilon$ and~$\d\Upsilon$.  An integral
element of~$(\cI,\Omega)$ is now cut out by equations of the form
\be
\omega_{ij} - \epsilon_{ijk}\,S_{kl}\,\omega_l = 0
\ee
where~$S = (S_{kl})$ is a $3$-by-$3$ symmetric matrix 
that satisfies~$\sigma_2(S)=\frac12 R$ (where~$\sigma_2(S)$ is the second elementary function of the eigenvalues of~$S$).  The characteristic variety
of such an integral element consists of the null (co-)vectors of the
quadratic form
\be
Q_S = \tr(S)\,g - S_{ij}\,\omega_i\,\omega_j\,.
\ee
In particular, except for the case~$S=0$ (which can only occur when~$R=0$),
this integral element is Cartan-regular, with Cartan character sequence
$(s_1,s_2,s_3) = (1,2,0)$.  In particular, the system~$(\cI,\Omega)$ is
involutive, and, in the real-analytic case, the general integral depends
on two functions of 2 variables.  

Note, by the way, that when~$Q_S$ is positive (or negative) definite,
the linearization around such a solution is elliptic and hence such
coclosed coframings are as regular as the metric~$g$.  In particular, 
this always happens when~$R>0$, i.e., when the scalar curvature is positive.%
\footnote{It is interesting to note that, when~$R>0$, the ideal~$\cI$ 
on the $6$-manifold~$B$ is algebraically equivalent at each point 
to the special Lagrangian ideal on~$\bbC^3$.}
Thus, for a real-analytic metric with positive scalar curvature, 
all of its coclosed coframings are real-analytic.
\end{remark}

\section{$\Gtwo$-manifolds}\label{sec: 7mfds}

For background on the group~$\Gtwo\subset\SO(7)$ and $\Gtwo$-manifolds, 
the reader can consult~\cite{MR916718,MR2282011,MR1787733,MR1004008}. 
I will generally follow the notation in~\cite{MR2282011}.

The crucial point is that the group $\Gtwo$ can be defined as the
stabilizer in~$\GL(7,\bbR)$ of the $3$-form
\be\label{eq: phi def}
\phi = e^{123}+e^{145}+e^{167}+e^{246}-e^{257}-e^{347}-e^{356}.
\ee
where~$e^{ijk}=e^i\w e^j\w e^k$ 
and the~$e^i$ are a basis of linear forms on~$\bbR^7$.  The Lie group~$\Gtwo$
is connected, has dimension~$14$, preserves the metric and orientation 
for which the~$e^i$ are an oriented orthonormal basis, 
and acts transitively on the unit sphere~$S^6\subset\bbR^7$.
The $\Gtwo$-stabilizer of~$e^1$ is the subgroup~$\SU(3)\subset\SO(6)$
that preserves the $2$-form~$e^{23}+e^{45}+e^{67}$ 
and the $3$-form~
\be
e^{246}-e^{257}-e^{347}-e^{356} 
= \Re\bigl((e^2+\iC\,e^3)\w(e^4+\iC\,e^5)\w(e^6+\iC\,e^7)\bigr).
\ee

The~$\GL(7,\bbR)$-orbit of~$\phi$ in~$\Lambda^3(\bbR^7)$ is an open 
(but not convex) cone~$\Lambda^3_+(\bbR^7)\subset\Lambda^3(\bbR^7)$
that consists precisely of the $3$-forms on~$\bbR^7$ whose stabilizers
are isomorphic to~$\Gtwo$.  Consequently, for any smooth~$7$-manifold~$M^7$,
there is a well-defined open subset~$\Lambda^3_+(T^*M)\subset\Lambda^3_+(T^*M)$ 
consisting of the $3$-forms whose stabilizers are isomorphic to~$\Gtwo$.

A $\Gtwo$-structure on~$M$ is thus specified by a $3$-form~$\sigma$ on~$M$
with the property that~$\sigma_x$ lies in~$\Lambda^3_+(T^*_xM)$ 
for all~$x\in M$.  Such a $3$-form~$\sigma$ will be said to be \emph{definite}.
Explicitly, the $\Gtwo$-structure~$B = B_\sigma$ consists of the 
linear isomorphisms~$u:T_x\to\bbR^7$ that satisfy~$u^*\phi = \sigma_x$.
Conversely, given a $\Gtwo$-structure~$B\to M$, 
there is a unique definite $3$-form~$\sigma$ on~$M$ such that~$B=B_\sigma$.
Put another way, the $\Gtwo$-structure bundle~$\cF(M)/\Gtwo$ is naturally
identified with~$\Lambda^3_+(M)$ by identifying~$[u]= u\cdot\Gtwo$ with
$u^*\phi\in \Lambda^3_+(T^*_{\pi(u)}M)$ for all~$u\in\cF(M)$.

A~$\sigma\in\Omega^3_+(M)$ determines a unique metric~$g_\sigma$
and orientation~$\ast_\sigma$ by requiring that the corresponding
coframings~$u\in B_\sigma$ be oriented isometries.
 
It is a fact~\cite{MR916718} that~$B_\sigma$ is torsion-free
if and only if~$\sigma$ is $g_\sigma$-parallel, 
which, in turn, holds if and only if
\be\label{eq: G2manifoldeqs}
\d\sigma= 0 \qquad\text{and}\qquad  \d({\ast}_\sigma\sigma) = 0.
\ee
Thus, a $\Gtwo$-manifold can be regarded as a pair~$(M^7,\sigma)$ 
where~$\sigma\in\Omega^3_+(M)$ satisfies
the nonlinear system of~\textsc{pde}~\eqref{eq: G2manifoldeqs}.

By a theorem of Bonan (see~\cite[Chapter X]{MR0867684}),
for any $\Gtwo$-manifold~$(M,\sigma)$, 
the associated metric~$g_\sigma$ has vanishing Ricci tensor.  
In particular, by a result of DeTurck and Kazdan~\cite{MR644518},
the metric~$g_\sigma$ is real-analytic in $g_\sigma$-harmonic coordinates.
Since~$\sigma$ is $g_\sigma$-parallel, it, too, must be real-analytic
in~$g_\sigma$-harmonic coordinates.

There is a natural differential ideal on~$\Lambda^3_+(T^*M) = \cF(M)/\Gtwo$
defined as follows:  For~$[u]\in\cF(M)/\Gtwo$, define
\be
\bfsigma_{[u]} = \pi^*(u^*\phi)
\qquad\text{and}\qquad
\bftau_{[u]} = \pi^*\bigl(u^*(\ast_\phi\phi)\bigr),
\ee
where~$\pi:\cF(M)/\Gtwo\to M$ is the natural basepoint projection.
Let~$\cI$ be the differential ideal on~$\cF(M)/\Gtwo = \Lambda^3_+(T^*M)$ 
generated by~$\d\bfsigma$ and~$\d\bftau$.
The following result is proved in~\cite{MR916718}:

\begin{theorem}\label{thm: G2involutivity} 
The ideal~$\cI$ on~$\Lambda^3_+(T^*M)$ is involutive. 
A section~$\sigma\in\Omega^3_+(M)$
is an integral of~$\cI$ and only if it is~$g_\sigma$-parallel.
Modulo diffeomorphisms, the general $\cI$-integral~$\sigma$ 
depends on $6$ functions of $6$ variables.
\end{theorem}

\subsection{Hypersurfaces} 
The group~$\Gtwo$ acts transitively on~$S^6\subset\bbR^7$, 
with stabilizer~$\SU(3)$.
Hence, an oriented~$N^6\subset M$ inherits a canonical $\SU(3)$-structure,
which is determined by the $(1,1)$-form~$\omega$ 
and~$(3,0)$-form~$\Omega = \phi + \iC\,\psi$ defined by
\be
\omega = \nb\lhk \sigma
\qquad\text{and}\qquad
\Omega = \phi + \iC\,\psi = N^*\sigma - \iC\,(\nb\lhk{\ast}_\sigma\sigma).
\ee

In fact, if one defines~$f:\bbR\times N\to M$ by
\be
f(t,p) = \exp_p\bigl(t\,\nb(p)\bigr),
\ee
then
\be
f^*\sigma = \d t \w \omega + \Re(\Omega)
\qquad\text{and}\qquad
f^*({\ast}_\sigma\sigma) = \tfrac12\,\omega^2 - \d t \w \Im(\Omega).
\ee
where, now,~$\omega$ and~$\Omega$ are forms on~$N$ that depend on~$t$.

For each fixed~$t=t_0$, 
the induced $\SU(3)$-structure on~$N$ satisfies
\be
\d\Re(\Omega)=\d(f^*_{t_0}\sigma) = 0
\qquad\text{and}\qquad
\d(\tfrac12\,\omega^2) = \d\bigl(f^*_{t_0}({\ast}_\sigma\sigma)\bigr) = 0,
\ee
so these are necessary conditions on the~$\SU(3)$-structure on~$N$
that it be induced by immersion into a $\Gtwo$-holonomy manifold~$M$.

\begin{theorem}\label{thm: SU3emedding}  
A real-analytic $\SU(3)$-structure on~$N^6$
is induced by embedding into a $\Gtwo$-manifold if and only if 
its defining forms~$\omega$ and~$\Omega$ satisfy
\be\label{eq: G2hypersurfaceconditions}
\d\Re(\Omega)=0 \qquad\text{and}\qquad
\d(\tfrac12\,\omega^2) = 0.
\ee
\end{theorem}

\begin{proof}
The necessity of~\eqref{eq: G2hypersurfaceconditions} 
has already been demonstrated, so I will just prove the sufficiency.

Define a tautological $2$-form~$\bfomega$ 
and~$3$-form~$\bfOmega$ on~$\cF(N)/\SU(3)$
as follows:  For a coframe~$u:T_xN\to\bbC^3$, 
define these forms at~$[u]=u\cdot\SU(3)\in\cF(N)/\SU(3)$ by
\be
\bfomega_{[u]} = \pi^*\bigl(u^*(\tfrac\iC2({}^t\d z\w\d\bar z))\bigr)
\qquad\text{and}\qquad
\bfOmega_{[u]} = \pi^*\bigl(u^*(\d z^1{\w}\d z^2{\w}\d z^3)\bigr)
\ee
where~$\pi:\cF(N)/\SU(3)\to N$ is the basepoint projection.

On~$X = \bbR\times\cF(N)/\SU(3)$, consider the $3$-form and~$4$-form
defined by
\be
\begin{aligned}
\bfsigma &= \d t \w \bfomega + \Re(\bfOmega)\\
\bftau &= \tfrac12\,\bfomega^2 - \d t \w \Im(\bfOmega).
\end{aligned}
\ee

Let~$\cI$ be the EDS generated by the closed $4$-form~$\d\bfsigma$ 
and~$5$-form~$\d\bftau$.  Then the calculation used to
prove Theorem~\ref{thm: G2involutivity} (see~\cite{MR916718})
shows that~$\cI$ is involutive, with characters
\be
(s_1,\ldots,s_7)=(0,0,1,4,10,13,0).
\ee

Since~$\d\bigl(\Re(\Omega)\bigr)=\d(\tfrac12\,\omega^2)=0$,
the given~$\SU(3)$-structure on~$N$ 
defines a regular integral manifold~$L\subset X$ 
of~$\cI$ lying in the hypersurface~$t=0$.  

Since the given~$\SU(3)$-structure is assumed to be real-analytic,
the system~$\cI$ and the integral manifold~$L\subset X$ are real-analytic
by constuction.  Hence the Cartan-K\"ahler Theorem can be applied to conclude
that $L$ lies in a unique~$\cI$-integral~$M^7\subset X$.  The pullback
of~$\bfsigma$ to~$M$ is then a closed definite $3$-form~$\sigma$ on~$M$
while the pullback of~$\bftau$ to~$M$ 
is closed and equal to~$\ast_\sigma\sigma$.  
Thus,~$(M,\sigma)$ is a $\Gtwo$-manifold.  
By construction,~$N$ is imbedded into~$M$ as the locus~$t=0$ 
and the $\SU(3)$-structure on~$N$ induced by~$\sigma$ is the given one.
\end{proof}

\begin{theorem}\label{thm: SU3nonembedding}
There exist non-real-analytic~$\SU(3)$-structures
on~$N^6$ whose associated forms~$(\omega,\Omega)$ 
satisfy~\eqref{eq: G2hypersurfaceconditions}
but that are not induced from an immersion 
into a $\Gtwo$-manifold~$(M,\sigma)$.

In fact, if a non-analytic $\SU(3)$-structure 
satisfies~\eqref{eq: G2hypersurfaceconditions} and
\be\label{eq: meancurvC}
\ast\bigl(\omega\w \d\bigl(\Im(\Omega)\bigr)\bigr) = C
\ee
for some constant~$C$, then it cannot be~$\Gtwo$-immersed.

Non-analytic $\SU(3)$-structures 
satisfying~\eqref{eq: G2hypersurfaceconditions} and \eqref{eq: meancurvC} 
do exist.
\end{theorem}

\begin{proof}
When an~$\SU(3)$-structure on~$N^6$ 
with defining forms~$(\omega,\Omega)$ is induced 
via a~$\Gtwo$-immersion~$N^6\hookrightarrow M^7$,
the mean curvature~$H$ of~$N$ in~$M$ is given by
\be
-12H = \ast\bigl(\omega\w \d\bigl(\Im(\Omega)\bigr)\bigr).
\ee
Thus, when the right hand side of this equation is constant, it follows
by elliptic regularity that~$N^6$ is a real-analytic submanifold
of the real-analytic~$(M^7,\sigma)$. 

Thus, if the given $\SU(3)$-structure on~$N$
satisfying~\eqref{eq: G2hypersurfaceconditions}
and~\eqref{eq: meancurvC} is not real-analytic, 
it cannot be induced by an embedding into a $\Gtwo$-holonomy
manifold~$M$.

It remains to construct a non-analytic example satisfying
\eqref{eq: G2hypersurfaceconditions} and~\eqref{eq: meancurvC}.
Here is why it is somewhat delicate:
Since~$\dim\bigl(\GL(6,\bbR)/\SU(3)\bigr) = 28$, 
a choice of an $\SU(3)$-structure~$(\omega,\Omega)$ on~$N^6$ 
depends on~$28$ functions of~$6$ variables. 
Modulo diffeomorphisms, this leaves~$22$ functions of $6$ variables. 
On the other hand, the equations
\be
\d\bigl(\Re(\Omega)\bigr)=0,\qquad
\d(\tfrac12\,\omega^2)=0,\qquad
\ast\bigl(\omega\w \d\bigl(\Im(\Omega)\bigr)\bigr) = C
\ee
constitute $15+6+1 = 22$ equations for the $\SU(3)$-structure.

Thus, the equations to be solved are `more determined' than
in the analogous $\SU(2)$ case.  Nevertheless, their diffeomorphism
invariance still allows one to construct the desired example, 
as will now be shown.

Say that a $3$-form~$\phi\in\Omega^3(N^6)$ 
is \emph{elliptic} if, at each point, 
it is linearly equivalent to~$\Re\bigl(\d z^1\w\d z^2\w\d z^3\bigr)$.
This is a open pointwise condition on~$\phi$ (i.e., it is
\emph{stable} in Hitchin's sense~\cite{MR1871001}): The elliptic
$3$-forms are sections of an open subbundle~$\Lambda^3_e(T^*N)
\subset\Lambda^3(T^*N)$.  I will denote the set of elliptic 
$3$-forms on~$N$ by~$\Omega^3_e(N)$.
  
Fix an orientation of~$N^6$.  
An elliptic~$\phi\in\Omega^3_e(N)$
then defines a unique, orientation-preserving 
almost-complex structure~$J_\phi$ on~$N^6$ such that
\be
\Omega_\phi = \phi + \iC\, J_\phi^*(\phi)
\ee
is of $J_\phi$-type~$(3,0)$.

Now assume that~$\phi\in\Omega^3_e(N)$ is closed.  
Then~$\d\Omega_\phi$ is purely imaginary 
and yet must be a sum of terms of $J_\phi$-type~$(3,1)$
and~$(2,2)$.  Thus, $\d\Omega_\phi$ is purely of $J_\phi$-type~$(2,2)$.

Let~$\Lambda^{1,1}_+(N,J_\phi)$ denote the set of real $2$-forms
that are of $J_\phi$-type~$(1,1)$ 
and that are positive on all $J_\phi$-complex lines.
The squaring map
$$
\sigma:\Lambda^{1,1}_+(N,J_\phi)\to\Lambda^{2,2}(N,J_\phi)
$$
given by~$\sigma(\omega) = \omega^2$ is a diffeomorphism 
onto the open set~$\Lambda^{2,2}_+(N,J_\phi)\subset\Lambda^{2,2}(N,J_\phi)$
that consists of the real $4$-forms of $J_\phi$-type~$(2,2)$ 
that are positive on all $J_\phi$-complex $2$-planes.

Now fix a constant~$C\not=0$.  One sees from the above
discussion that it is a $C^1$-open condition 
on~$\phi$ that
\be
\d\Omega_\phi = \tfrac\iC6 C\,(\omega_\phi)^2
\qquad\text{for some}\qquad
\omega_\phi = \overline{\omega_\phi} \in \Omega^{1,1}_+(N,J_\phi).
\ee

Now, the pair~$(\omega_\phi,\Omega_\phi)$ are the defining forms of
an $\SU(3)$-structure on~$N$ if and only if
\be\label{eq: volumeequality}
\tfrac16(\omega_\phi)^3 
- \tfrac18\iC\, \Omega_\phi\w\overline{\Omega_\phi} = 0.
\ee
This is a single, first-order scalar equation on the closed $3$-form~$\phi$.
It is easy to see that there are non-analytic solutions.  (For example,
if one starts with the standard structure induced on the $6$-sphere in
flat~$\bbR^7$ endowed with its flat $\Gtwo$-structure, 
then small perturbations of the corresponding closed~$\phi$ 
can be made that solve~\eqref{eq: volumeequality} but for which
the induced $\SU(3)$-structure has constant curvature on a proper open subset of~$S^6$. Such an $\SU(3)$-structure clearly cannot be real-analytic everywhere.)

Assuming~\eqref{eq: volumeequality} is satisfied, we have
\be
\d(\Re\Omega_\phi) = \d\phi = 0,
\ee
and
\be
\d\bigl(\tfrac12(\omega_\phi)^2\bigr) 
= \d\bigl(-3\iC\,\tfrac1C\,\d\Omega_\phi\bigr)
= 0,
\ee
and, finally
\be
{\ast}_\phi\bigl(\omega_\phi\w\d(\Im\Omega_\phi)\bigr)
= {\ast}_\phi\bigl(\omega_\phi\w \tfrac16 C\,(\omega_\phi)^2\bigr) = C.
\ee
\end{proof}

\subsection{The flow interpretation} 
On~$N^6\times\bbR$, 
with~$(\omega,\Omega)$ defining an $\SU(3)$-structure
on~$N^6$ depending on~$t\in\bbR$, consider the equations
\be\label{eq: Gtwofloweqs}
\d\bigl(\d t \w \omega + \Re(\Omega)\bigr) = 0
\qquad\text{and}\qquad
\d\bigl(\tfrac12\,\omega^2 - \d t \w \Im(\Omega)\bigr)=0,
\ee
which assert that~$\sigma=\d t \w \omega + \Re(\Omega)$,
which is a definite $3$-form on~$M=N^6\times\bbR$, is
both closed and co-closed (since~$\ast_\sigma\sigma
= \tfrac12\,\omega^2 - \d t \w \Im(\Omega)$).

Think of~$\Omega$ as $\phi + \iC J_\phi^*(\phi)$, 
so that the $\SU(3)$-structure is determined by~$(\omega,\phi)$ 
where~$\phi=\Re(\Omega)$.  
The conditions~\eqref{eq: Gtwofloweqs} for fixed~$t$ are then
\be\label{eq: GtwoICs}
\d\phi = 0 \qquad\text{and}\qquad\d(\omega^2)=0,
\ee
while the $\Gtwo$-evolution equations implied by~\eqref{eq: Gtwofloweqs} 
for such~$(\omega,\phi)$ are then
\be\label{eq: Gtwoevol}
\frac{\d\hfill}{\d t}(\phi) = \d\omega
\qquad\text{and}\qquad
\frac{\d\hfill}{\d t}(\omega) 
= -{L_\omega}^{-1}\left(\d\bigl(J_\phi^*(\phi)\bigr)\right),
\ee
where~$L_\omega:\Omega^2(N)\to\Omega^4(N)$ 
is the invertible map~$L_\omega(\beta) = \omega\w\beta$.

Theorems~\ref{thm: SU3emedding} and \ref{thm: SU3nonembedding} 
show that solutions to the `$\Gtwo$-flow'~\eqref{eq: Gtwoevol} 
do exist for analytic initial $\SU(3)$-structures 
satisfying the closure conditions~\eqref{eq: GtwoICs},
but need not exist for non-analytic initial $\SU(3)$-structures 
satisfying these closure conditions.

\section{$\Spin(7)$-manifolds}\label{sec: 8mfds}

For background on the group~$\Spin(7)\subset\SO(8)$ and $\Spin(7)$-manifolds, 
the reader can consult~\cite{MR916718,MR1787733,MR1004008}. 
I will generally follow the notation in~\cite{MR916718}.

The main point is that the group~$\Spin(7)\subset\SO(8)$ 
is the $\GL(8,\bbR)$-stabilizer of the $4$-form~$\Phi_0\in\Lambda^4(\bbR^8)$,
defined by
\be\label{eq: Phi def}
\Phi_0 = e^0\w\phi + \ast_\phi\phi
\ee
where~$\phi$ is defined by~\eqref{eq: phi def} and $e^0,\ldots,e^7$
is a basis of linear forms on~$\bbR^8 = \bbR\oplus\bbR^7$.  
The group~$\Spin(7)$ is a connected Lie group of dimension~$21$ 
that is the double cover of~$\SO(7)$.  It acts transitively on the unit
sphere in~$\bbR^8$ and the $\Spin(7)$-stabilizer of~$e^0$ is~$\Gtwo$. 
The~$\GL(8,\bbR)$-orbit of~$\Phi_0$ in~$\Lambda^4(\bbR^8)$ 
will be denoted by~$\Lambda^4_s(\bbR^8)$.  This orbit has dimension~$43$,
so it is not open in~$\Lambda^4(\bbR^8)$.

A $\Spin(7)$-structure on~$M^8$ 
is thus specified by a $4$-form~$\Phi\in\Omega^4(M)$
that is linearly equivalent to~$\Phi_0$ at each point of~$M$.
The set of such $4$-forms will be denoted~$\Omega^4_s(M)$.  
They are the sections of the bundle~$\cF(M)/\Spin(7)\to M$, which
has a natural embedding into~$\Lambda^4(T^*M)$.

Given a $4$-form~$\Phi\in\Omega^4_s(M)$, 
the corresponding~$\Spin(7)$-structure~$B = B_\Phi$
is defined to be the set of coframings~$u:T_xM\to\bbR^8$ that satisfy
$u^*\Phi_0 = \Phi_x$.  Conversely, every $\Spin(7)$-structure~$B\to M$
is of the form~$B=B_\Phi$ for a unique $4$-form~$\Phi\in\Omega^4_s(M)$.  

In particular, each~$\Phi\in\Omega^4_s(M)$ 
determines a metric~$g_\Phi$ and orientation~${\ast}_\Phi$ on~$M$
by requiring that the elements~$u\in B_\Phi$ be oriented isomorphisms.

It is a fact~\cite{MR916718} that~$B_\Phi$ is torsion-free
if and only if~$\Phi$ is $g_\Phi$-parallel, 
which, in turn, holds if and only if
\be\label{eq: Spin7manifoldeqs}
\d\Phi= 0.
\ee
Thus, a $\Spin(7)$-manifold can be regarded as a pair~$(M^8,\Phi)$ 
where~$\Phi\in\Omega^4_s(M)$ satisfies
the nonlinear system of~\textsc{pde}~\eqref{eq: Spin7manifoldeqs}.

By a theorem of Bonan (see~\cite[Chapter X]{MR0867684}),
for any $\Spin(7)$-manifold~$(M,\Phi)$, 
the associated metric~$g_\Phi$ has vanishing Ricci tensor.  
In particular, by a result of DeTurck and Kazdan~\cite{MR644518},
the metric~$g_\Phi$ is real-analytic in $g_\Phi$-harmonic coordinates.
Since~$\Phi$ is $g_\Phi$-parallel, it, too, must be real-analytic
in~$g_\Phi$-harmonic coordinates.

Define a $4$-form~$\bfPhi$ on~$\cF(M)/\Spin(7)$ by the following rule:  
For $u:T_xM\to\bbR^8$ and~$[u]=u\cdot\Spin(7)$, set
\be
\bfPhi_{[u]} = \pi^*\bigl(u^*\Phi_0\bigr)
\ee
where~$\pi:\cF(M)\to M$ is the basepoint projection.  Let~$\cI$ be
the ideal generated by~$\d\bfPhi$ on~$\cF(M)/\Spin(7)$.  The following
result is proved in~\cite{MR916718}:

\begin{theorem}\label{thm: Spin7involutive} 
A section~$\Phi\in\Omega^4_s(M)$ satisfies
\eqref{eq: Spin7manifoldeqs} if and only if it is an integral of~$\cI$.
The ideal $\cI$ is involutive.  Modulo diffeomorphisms,
the general $\cI$-integral~$\Phi$ depends on $12$ functions of $7$ variables.
\end{theorem}

\subsection{Hypersurfaces}
$\Spin(7)$ acts transitively on~$S^7$
and the stabilizer of a point is~$\Gtwo$.
An oriented hypersurface~$N^7\subset M^8$ 
inherits a $\Gtwo$-structure~$\sigma\in\Omega^3_+(M)$
that is defined by the rule
\be
\sigma = \nb\lhk\Phi
\ee
where~$\nb$ is the oriented normal vector field along~$N$.
It also satisfies
\be
{\ast}_\sigma\sigma = N^*\Phi.
\ee

The structure equations show that
\be
{\ast}_\sigma\bigl(\sigma\w\d\sigma\bigr) = 28H
\ee
where~$H$ is the mean curvature of~$N$ in~$(M,g_\Phi)$.

\begin{theorem}\label{thm: G2embedding} 
If~$\sigma\in\Omega^3_+(N^7)$ is real-analytic
and satisfies~$\d({\ast}_\sigma\sigma)=0$, then~$\sigma$ is induced
by an immersion of~$N$ into a $\Spin(7)$-manifold~$(M,\Phi)$. 
\end{theorem}

\begin{proof}
The argument in this case is entirely analogous to the $\SU(2)$
and~$\Gtwo$ cases already treated:

Recall the definitions of~$\bfsigma$ and~$\bftau$ on~$\cF(N^7)/\Gtwo$
and, on $X = \bbR\times\cF(N)/\Gtwo$, define
\be
\bfPhi = \d t\w \bfsigma + \bftau.
\ee
Let~$\cI$ be the ideal on~$\bbR\times\cF(N)/\Gtwo$ generated by~$\d\bfPhi$.
The same calculation used to prove Theorem~\ref{thm: Spin7involutive}
(see~\cite{MR916718}) then yields that~$\cI$ is involutive, with
character sequence
\be
(s_1,\ldots,s_8)=(0,0,0,1,4,10,20,0).
\ee
Since~$\d(\ast_\sigma\sigma) = 0$, the $\Gtwo$-structure~$\sigma$ 
defines a regular $\cI$-integral~$L\subset X$ within the locus~$t=0$.

Since $\sigma$ is assumed to be real-analytic with respect to some underlying
analytic structure on~$N$, it follows that~$X$,~$\cI$, 
and~$L$ are real-analytic with respect to the obvious induced analytic structures on the appropriate underlying manifolds.
Thus, the Cartan-K\"ahler theorem applies to show that~$L$ lies in
a unique $\cI$-integral~$M^8\subset X$.  The form~$\bfPhi$ then pulls back
to~$M$ to be a closed $\Phi\in\Omega^4_s(M)$ which induces the given~$\sigma$
on~$N\simeq L\subset M$ defined by~$t=0$.
\end{proof}

\begin{theorem}\label{thm: G2nonembedding} 
There exist non-real-analytic
$\Gtwo$-structures~$\sigma\in\Omega^3_+(N^7)$ 
that satisfy
\be\label{eq: sigmacoclosed}
\d({\ast}_\sigma\sigma)=0
\ee
but that are not induced from a $\Spin(7)$-immersion.

In fact, if a non-analytic $\Gtwo$-structure 
satisfies~\eqref{eq: sigmacoclosed} and
\be\label{eq: G2constmeancurv}
\ast_\sigma\bigl(\sigma\w \d\sigma\bigr) = C
\ee
where~$C$ is a constant, then it cannot be~$\Spin(7)$-immersed.

Non-analytic $\Gtwo$-structures~$\sigma\in\Omega^3_+(N^7)$
satisfying~\eqref{eq: sigmacoclosed} and~\eqref{eq: G2constmeancurv} do exist.
\end{theorem}

\begin{proof} 
The first claim follows from the second and third.

If~$\sigma\in\Omega^3_+(N^7)$ is induced from at $\Spin(7)$-immersion
$(N,\sigma)\hookrightarrow (M^8,\Phi)$
and has $\ast_\sigma\bigl(\sigma\w \d\sigma\bigr)$ equal to a constant,
then the hypersurface~$N\subset M$ has constant mean curvature.  
Since~$(M,g_\Phi)$ is real-analytic, constant mean curvature hypersurfaces
in~$M$ are also real-analytic, so it follows that~$\sigma$ must be real-analytic.  This proves the second claim.

It remains now to verify the third claim by explaining how to construct
non-analytic $\Gtwo$-structures~$\sigma$ satisfying
\eqref{eq: sigmacoclosed} and~\eqref{eq: G2constmeancurv}. 
To save space, I will only sketch the argument, 
the details of which are somewhat involved, though
the idea is the same as for~$\SU(2)$:  

If such a $\sigma$ is to be real-analytic, 
it will have to be real-analytic in $g_\sigma$-harmonic coordinates.  
Now, the system of first-order equations
$$
\d({\ast}_\sigma\sigma)=0,
\qquad\qquad
\ast_\sigma\bigl(\sigma\w \d\sigma\bigr) = C,
\qquad\qquad
\d(\ast_\sigma\d x) = 0
$$
for~$\sigma\in\Omega^3_+(\bbR^7)$ is only $21+1+7=29$
equations for~$35$ unknowns.  This underdetermined system is not elliptic, 
but its symbol mapping has constant rank 
and it can be embedded into an appropriate sequence 
to show that it has non-real-analytic solutions.
\end{proof}

\subsection{The flow interpretation}
Finally, let us consider the `flow' interpretation:  
If~$f:N\hookrightarrow M$ 
is an oriented smooth hypersurface in a $\Spin(7)$-manifold~$(M,\Phi)$
with oriented unit normal~$\nb:N\to TM$, then the normal exponential
mapping can be used to embed a neighborhood~$U$ of~$N$ in~$M$ 
into~$\bbR\times N$ in such a way that, on this neighborhood~$U$, one can
write
\be
\Phi = \d t \w \sigma + \ast_\sigma\sigma
\ee
where~$\sigma\in\Omega^3_+(N)$ now depends on~$t$ (in the case that~$N$
is noncompact, the domain of~$\sigma$ might depend on~$t$).  The closure
of~$\Phi$ implies that
\be
\frac{\d\hfill}{\d t}\bigl(\ast_\sigma\sigma\bigr) = \d \sigma.
\ee 

This equation is enough to determine the time derivative of~$\sigma$
as well because of the following observation:  For any real vector space~$V$ 
of dimension~$7$, the map~$\Ss:\Lambda^3_+(V^*)\to\Lambda^4(V^*)$ defined
by~$\Ss(\phi) = \ast_\phi\phi$ is a $2$-to-$1$ smooth local diffeomorphism 
of~$\Lambda^3_+(V^*)$ onto an open cone~$\Lambda^4_+(V^*)\subset\Lambda^4(V^*)$.
In fact, if~$\Ss(\phi)=\Ss(\psi)$, then~$\phi = \pm \psi$.  Both $\phi$
and~$-\phi$ are definite, but they each determine opposite orientations, 
i.e., $\ast_{-\phi}1=-\ast_\phi1$.  In particular, 
if one fixes an orientation on~$V$, then for any~$\tau\in \Lambda^4_+(V^*)$,
there is a unique element~$\phi\in\Lambda^3_+(V^*)$ such that $\ast_\phi\phi
= \tau$ and $\ast_\phi1$ is a positive volume form on~$V$.  I will denote this element by~$\Ss^{-1}(\tau) = \phi$.  Using this notation and the assumption
that~$N$ is oriented, the above equation can be written in the more obviously
`evolutionary' form
\be\label{eq: tauevolves}
\frac{\d\hfill}{\d t}\bigl(\tau\bigr) = \d \left(\Ss^{-1}(\tau)\right)
\ee
where
\be\label{eq: Phiastau}
\Phi = \d t \w \Ss^{-1}(\tau) + \tau
\ee
with~$\tau\in\Omega^4_+(N)$ depending on~$t$.

The content of Theorems~\ref{thm: G2embedding} and \ref{thm: G2nonembedding}
is then that~\eqref{eq: tauevolves} has a solution when the initial~$\tau_0$
is closed and real-analytic, but need not have a solution for a non-analytic
closed~$\tau_0$. 

\bibliographystyle{hamsplain}

\begin{thebibliography}{10}

\bibitem{MR0867684}
A. Besse, 
\emph{Einstein manifolds}, 
Ergebnisse der Mathematik und ihrer Grenzgebiete (3) 
[Results in Mathematics and Related Areas (3)], 10. 
Springer-Verlag, Berlin, 1987. 
MR~0867684 (88f:53087) 

\bibitem{MR48133} 
A. Borel and A. Lichnerowicz.
\emph{Groupes d'holonomie des vari\'et\'es riemann\-iennes}, 
C.R.\ Acad.\ Sci.\ Paris \textbf{234} (1952), 1835--1837.
MR~48133 (13,986b)

\bibitem{MR916718}
R. Bryant, 
\emph{Metrics with exceptional holonomy},
Ann. of Math. (2) \textbf{126} (1987), 525--576.  
MR~916718 (1989b:53084)

\bibitem{MR2282011}
\bysame,
\emph{Some remarks on $\Gtwo$-structures}, 
Proceedings of G\"okova Geometry-Topology Conference 2005, 75--109, 
G\"okova Geometry/Topology Conference (GGT), G\"okova, 2006. 
MR~2282011 (2007k:53019)  arXiv:math.DG/0305124

\bibitem{MR1083148}
R. Bryant {\it et al.\/}, 
\emph{Exterior Differential Systems},
MSRI Series \textbf{18}, Springer-Verlag, 1991.
MR~1083148 (1992h:58007)

\bibitem{Cartan1925} 
\'E.\ Cartan, 
\emph{La G\'eometrie des Espaces de Riemann}, 
M\'emorial des Sciences Mathematiques, Fasc.~IX (1925).

\bibitem{MR644518}
D. DeTurck and J. Kazdan,
\emph{Some regularity theorems in Riemannian geometry},
Ann. Sc. \'Ec. Norm. Sup.~\textbf{14} (1981), 249--260.
MR~644518 (83f:53018)

\bibitem{MR1871001}
N. Hitchin,
\emph{Stable forms and special metrics}. 
Global differential geometry: the mathematical legacy of Alfred Gray 
(Bilbao, 2000), 70--89, Contemp. Math., \textbf{288}, 
Amer. Math. Soc., Providence, RI, 2001.
MR~1871001 (2003f:53065) 

\bibitem{MR1787733}
D. Joyce, 
\emph{Compact manifolds with special holonomy}, 
Oxford Mathematical Monographs, 
Oxford University Press, Oxford, 2000.
MR~1787733 (2001k:53093)

\bibitem{MR1004008}
S. Salamon, 
\emph{Riemannian geometry and holonomy groups}, 
Pitman Research Notes in Math., \textbf{201}.
Longman, Harlow, 1989.
MR~1004008 (90g:53058)

\bibitem{schouten1918} 
J. Schouten,
\emph{On the number of degrees of freedom of the geodetically moving systems}, 
Proc.\ Kon.\ Acad.\ Wet.\ Amsterdam\ \textbf{21} (1918), 607--613.

\end{thebibliography}

\providecommand{\bysame}{\leavevmode\hbox to3em{\hrulefill}\thinspace}

\end{document}